\documentclass[amstex,16pt, amssymb]{article}
\usepackage{mathtext}
\usepackage[cp1251]{inputenc}
\usepackage[T2A]{fontenc}
\usepackage[dvips]{graphicx}
\usepackage{amsmath}
\usepackage{amssymb}
\usepackage{amsxtra}
\usepackage{latexsym}
\usepackage{ifthen}
\usepackage{amsthm}

\def\XXint#1#2#3{{\setbox0=\hbox{$#1{#2#3}{\int}$}
     \vcenter{\hbox{$#2#3$}}\kern-.5\wd0}}

\theoremstyle{plain}
\newtheorem{theorem}{Theorem}[section]
\newtheorem{lemma}{Lemma}[section]
\newtheorem{proposition}{Proposition}[section]
\newtheorem{corollary}{Corollary}[section]

\theoremstyle{definition}

\newtheorem{remark}{Remark}[section]

\renewcommand{\abstract}{\textbf{Abstract. }\medskip}

\numberwithin{equation}{section}

\numberwithin{equation}{section}

\begin{document}

\title{Extremal problem of area approximation}

\author{O. O.~Pokutnyi, R. R.~Salimov, M. V. Stefanchuk}

\date{}

\maketitle

\begin{abstract}
This paper investigates the problem of optimal piecewise li\-near approximation of a smooth curve, in which the area of the region enclosed between the graph of the original function and the constructed interpola\-ting polyline is minimized. This problem arises in technological processes such as laser and plasma cutting of metals, milling, additive manufactu\-ring, and in the preparation of control programs for CNC equipment. For a quadratic parabola, the exact value of the minimum area is established, and it is shown that the optimal partition of the interval is uniform. For a cubic function, it is demonstrated that the optimal interpolation nodes are distributed non-uniformly; in the case of two internal partition points, their exact coordinates and the corresponding minimum area are obtained. The results can be applied in CAD systems for optimizing tool paths based on the criterion of minimum deviation area.
\end{abstract}

\bigskip


\bigskip



\section{Introduction}
Imagine you need to cut a part with a curved edge from a sheet. The cutter cannot move along a perfect curve — it moves in straight segments, and the programmer can only specify a limited number of points where the cutter changes direction. The task is to choose these points so that the gap between the ideal edge and the actual cutting line is as small as possible. This allows saving material, time, tooling, and electrical energy. Such a problem arises wherever flat parts with curved contours are cut or machined: in laser and plasma cutting of metals, CNC milling, 3D printing, and the preparation of control programs.

Problems of optimal piecewise approximation have a long history dating back to Bellman's dynamic programming approach and the minimax theory de\-ve\-lo\-ped by M.\,M. Gavrilovic (see \cite{Gavrilovic}). At the same time, the theory of opti\-mal appro\-ximation of function classes has been extensively developed by the Ukrainian approximation school represented by A.\,S. Romanyuk (see \cite{Romanyuk}) and their collaborators. However, the problem of minimizing the enclosed area between the graph of a smooth function and its interpolating polygonal line appears to have received little attention in the literature.

In this paper, we study the problem of finding the optimal partition of a segment for quadratic and cubic functions such that the area between the curve (the graph of the function) on this segment and the piecewise linear function passing through the partition points is minimal. In particular, we show that the optimal partitions for the functions $y = x^{2}$ and $y = x^{3}$ on the segment $[0,1]$ differ. In subsequent research, we will consider the corresponding problem of minimizing the $n$-dimensional volume.

Piecewise linear interpolation is the simplest representative of spline appro\-ximation and remains an important tool in approximation theory due to its computational efficiency and theoretical significance. Classical results on spline approximation, interpolation, and optimal approximation were developed in the works of J.\,H. Ahlberg, E.\,N. Nilson, J.\,L. Walsh, N.\,P.~Korneichuk, V.\,K.~Dzya\-dyk, A.\,S. Serdyuk, A.\,L. Shidlich, Carl de Boor, Larry L. Schumaker (see \cite{Ahlberg,Kornei,Dzyad,Serdyuk,Boor,Larry}). These contributions established the theoretical foundations of spline approximation and stimulated extensive research on interpolation methods and approximation processes.

While the classical theory primarily investigates approximation with respect to the uniform, integral, or other standard norms, considerably less attention has been paid to optimization criteria based on the geometric relationship between a function and its interpolant. In the present paper, we study the problem of minimizing the area enclosed between the graph of a function and its piecewise linear interpolant by optimally choosing the interpolation nodes. This formulation naturally combines ideas from spline approximation theory with nonlinear optimization and provides a new perspective on the construction of optimal piecewise linear interpolants.

\medskip

To solve the extremal problems, we will use the following result, see, e.g., \cite{Bertsekas,Boyd,Rockafellar}.

\begin{proposition}\label{prop1}
{\it Let \( F: \Omega \to \mathbb{R} \) be a twice continuously differentiable function
on a convex open set \( \Omega \subset \mathbb{R}^n \).
Suppose that:\\
1. \( x^* \in \Omega \) is a critical point, i.e. \( \nabla F(x^*) = 0; \)\\
2. the Hessian matrix \( H(x) \) is positive definite for all \( x \in \Omega. \)\\
Then \( x^* \) is the unique global minimum of \( F \) on \( \Omega \).}
\end{proposition}

\section{The case of a quadratic parabola}

We now turn to the problem of optimal piecewise linear approximation
for specific functions. The simplest nontrivial case is that of a
quadratic parabola, for which the area between the curve and the
interpolating polygonal line can be computed explicitly. 

Before proceeding to the main theorem, we establish a number of auxiliary results. The following lemma provides an explicit expression for the secant line of a quadratic parabola on an arbitrary interval.

\begin{lemma}\label{lemma1}
Let $f(x) = Ax^2 + Bx + C$, $A\neq 0$, and let $p < q$. The secant line connecting the points $(p, f(p))$ and $(q, f(q))$ on the graph of $f$ has the equation
\[
\ell_{p,q}(x) = \bigl(A(p+q) + B\bigr)x - A pq + C.
\]
\end{lemma}

\begin{proof}
The equation of the line through two points is
\[
\ell_{p,q}(x) = f(p) + \frac{f(q)-f(p)}{q-p}(x-p).
\]
Since $f(q)-f(p) = (q-p)(A(p+q)+B)$, substitution and simplification yield
\[
\ell_{p,q}(x) = Ap^2 + Bp + C + (A(p+q)+B)(x-p) = (A(p+q)+B)x - A pq + C.
\]
\end{proof}

The following lemma provides an explicit expression for the area between the graph of a quadratic parabola and the secant line on an arbitrary interval.

\begin{lemma}\label{lemma2}
Let $f(x) = Ax^2 + Bx + C$, $A \neq  0$, and let $p < q$. Then the area $S(p, q)$ between the graph of $f(x)$ and the secant line connecting the points $(p, f(p))$ and $(q, f(q))$ equals
\[
S(p, q) = \frac{|A|}{6}(q - p)^3.
\]
\end{lemma}

\begin{proof}
By Lemma \ref{lemma1},
\[
\ell_{p,q}(x) = \bigl(A(p+q) + B\bigr)x - A pq + C.
\]
Then
\[
S(p, q) = \left|\int_p^q \left( \bigl(A(p+q) + B\bigr)x - A pq + C \right) dx - \int_p^q \left( Ax^2 + Bx + C \right) dx\right|.
\]
Evaluating the integrals and collecting like terms yield
\[
S(p, q) = \frac{|A|}{6}(q - p)^3.
\]
\end{proof}

Now we establish a useful expression for the total area $S(x_1, \ldots, x_n)$ in terms of the lengths of the subintervals.

\begin{lemma}\label{lemma3}
Let $f(x) = Ax^2 + Bx + C$, $A \neq 0$, and let $a =x_0 < x_1 < \cdots < x_n < x_{n+1} = b$ be a partition. Let $\ell_i(x)$ be the chord connecting the points $(x_{i-1}, f(x_{i-1}))$ and $(x_i, f(x_i))$. Then the total area $S(x_1, \ldots, x_n)$ between the graph of $f(x)$ and the polygonal line formed by these chords is given by
\[
S(x_1, \ldots, x_n) = \frac{|A|}{6} \sum_{i=1}^{n+1} (x_i - x_{i-1})^3.
\]
\end{lemma}

\begin{proof}
By Lemma \ref{lemma2} applied to each subinterval $[x_{i-1}, x_i]$, the area between the parabola and the chord is
\[
S_i = \frac{|A|}{6} (x_i - x_{i-1})^3.
\]
Summing over all intervals gives
\[
S(x_1, \ldots, x_n) = \sum_{i=1}^{n+1} S_i = \frac{|A|}{6} \sum_{i=1}^{n+1} (x_i - x_{i-1})^3.
\]
\end{proof}


To prove the uniqueness and global nature of the found minimum, we need to study the matrix of second derivatives of the objective function. The following lemma gives the explicit form of the Hessian matrix for the quadratic case.

\begin{lemma}\label{lemma4}
Let
\[
S(x_1, \dots, x_n) = \frac{|A|}{6} \sum_{i=1}^{n+1} (x_i - x_{i-1})^3, \qquad x_0 = a, \quad x_{n+1} = b, \quad A\neq 0.
\]
Then the Hessian matrix
$
H = \left( \frac{\partial^2 S}{\partial x_i \partial x_j} \right)_{i,j=1}^{n}
$
of size \( n \times n \) has the following form:
\[
H = |A| \cdot
\begin{pmatrix}
x_2 - x_0 & -(x_2 - x_1) & 0 & \cdots & 0 \\
-(x_2 - x_1) & x_3 - x_1 & -(x_3 - x_2) & \cdots & 0 \\
0 & -(x_3 - x_2) & x_4 - x_2 & \cdots & 0 \\
\vdots & \vdots & \vdots & \ddots & \vdots \\
0 & 0 & 0 & \cdots & x_{n+1} - x_{n-1}
\end{pmatrix}.
\]

\end{lemma}

\begin{proof}
By Lemma \ref{lemma3}, we have
\[
S(x_1, \dots, x_n) = \frac{|A|}{6} \sum_{k=1}^{n+1} (x_k - x_{k-1})^3.
\]
Only two terms in this sum depend on \( x_i \) (those with indices \( k = i \) and \( k = i+1 \)):
\[
S = \frac{|A|}{6} \Bigl( (x_i - x_{i-1})^3 + (x_{i+1} - x_i)^3 \Bigr) + \left( S-\frac{|A|}{6} \Bigl( (x_i - x_{i-1})^3 + (x_{i+1} - x_i)^3 \Bigr)\right).
\]
Compute the first derivative with respect to $x_i$ of the function $S$:
\[
\frac{\partial S}{\partial x_i}
= \frac{|A|}{2} \Bigl( (x_i - x_{i-1})^2 - (x_{i+1} - x_i)^2 \Bigr).
\]
Next, we find the second derivatives of $S$.

Diagonal entry:
\[
\frac{\partial^2 S}{\partial x_i^2}
= |A| \bigl(  x_{i+1} - x_{i-1}  \bigr).
\]

Off-diagonal entry \( (i, i+1) \):
\[
\frac{\partial^2 S}{\partial x_i \partial x_{i+1}} = \frac{\partial}{\partial x_{i+1}} \left( \frac{|A|}{2} \bigl( (x_i - x_{i-1})^2 - (x_{i+1} - x_i)^2 \bigr) \right)
= -|A| (x_{i+1} - x_i).
\]

Entries \( (i, i-1) \): by symmetry of the Hessian matrix,
\[
\frac{\partial^2 S}{\partial x_i \partial x_{i-1}} = \frac{\partial^2 S}{\partial x_{i-1} \partial x_i} = -|A| (x_i-x_{i-1}).
\]

All other entries: if \( |i - j| > 1 \), then \( x_i \) and \( x_j \) appear in different terms of the sum, hence,
\[
\frac{\partial^2 S}{\partial x_i \partial x_j} = 0.
\]

Thus, the Hessian matrix has the form
\[
H = |A| \cdot
\begin{pmatrix}
x_2 - x_0 & -(x_2 - x_1) & 0 & \cdots & 0 \\
-(x_2 - x_1) & x_3 - x_1 & -(x_3 - x_2) & \cdots & 0 \\
0 & -(x_3 - x_2) & x_4 - x_2 & \cdots & 0 \\
\vdots & \vdots & \vdots & \ddots & \vdots \\
0 & 0 & 0 & \cdots & x_{n+1} - x_{n-1}
\end{pmatrix}.
\]

Lemma \ref{lemma4} is proved. 
\end{proof}

The following lemma shows that the Hessian matrix of \( S \) is positive definite for any partition of the interval.

\begin{lemma}\label{lemma5}
Let
\[
S(x_1, \dots, x_n) = \frac{|A|}{6} \sum_{i=1}^{n+1} (x_i - x_{i-1})^3, \qquad x_0 = a, \quad x_{n+1} = b, \quad A \neq 0.
\]
Then the matrix
\[
H = |A| \cdot
\begin{pmatrix}
x_2 - x_0 & -(x_2 - x_1) & 0 & \cdots & 0 \\
-(x_2 - x_1) & x_3 - x_1 & -(x_3 - x_2) & \cdots & 0 \\
0 & -(x_3 - x_2) & x_4 - x_2 & \cdots & 0 \\
\vdots & \vdots & \vdots & \ddots & \vdots \\
0 & 0 & 0 & \cdots & x_{n+1} - x_{n-1}
\end{pmatrix}
\]
is positive definite.
\end{lemma}

\begin{proof} Indeed, for an arbitrary nonzero vector \( \xi = (\xi_1, \ldots, \xi_n)^{T} \in \mathbb{R}^n \), the quadratic form is given by
\[
\xi^T H \xi = |A| \left( (x_1 - x_0) \xi_1^2 + \sum_{i=2}^n (x_i - x_{i-1})(\xi_i - \xi_{i-1})^2 + (x_{n+1} - x_n) \xi_n^2 \right),
\]
where all coefficients satisfy \( |A| > 0 \) and \( x_i - x_{i-1} > 0 \); hence each term is nonnegative. If \( \xi \neq 0 \), then at least one of these terms is strictly positive (otherwise, we successively obtain \( \xi_1 = 0, \xi_2 = 0, \ldots, \xi_n = 0 \), contradicting \( \xi \neq 0 \)). Therefore, \( \xi^T H \xi > 0 \), and the matrix \( H \) is positive definite. 
\end{proof}

The following
theorem shows that the optimal nodes are uniformly distributed and
gives an explicit formula for the minimum area.

\begin{theorem}\label{theorem1}
{\it Consider the segment $[a, b]$ with the corresponding parabola $y = Ax^2 + Bx + C$.
In this case, the optimal partition has the form
$$
x_i = x_0 + i \frac{x_{n+1} - x_0}{n+1}, \qquad
x_0 = a, \quad x_{n+1} = b, \quad i = 0, 1, \ldots, n, n+1,
$$
and the minimum area is
$$
S_{\min} = \frac{|A|}{6(n+1)^2}(b-a)^3.
$$}
\end{theorem}

\begin{proof}
Consider an arbitrary partition of the interval \([a, b]\):
\[
a = x_0 < x_1 < \dots < x_n < x_{n+1} = b.
\]

By Lemma \ref{lemma3}, the total area between the parabola and the polygonal line is
\begin{equation}\label{pr1th1}
S(x_1, \dots, x_n) =  \frac{|A|}{6} \sum_{i=1}^{n+1} (x_i - x_{i-1})^3.
\end{equation}

We now minimize \(S(x_1, \dots, x_n)\) subject to
\[
a = x_0 < x_1 < \dots < x_n < x_{n+1} = b.
\]
Computing the partial derivatives gives
\[
\frac{\partial S}{\partial x_i} = \frac{|A|}{2} \left( (x_i - x_{i-1})^2 - (x_{i+1} - x_i)^2 \right), \qquad i = 1, \dots, n.
\]
Setting them to zero yields
\[
(x_i - x_{i-1})^2 = (x_{i+1} - x_i)^2.
\]
Since all differences are positive, we obtain
\[
x_i - x_{i-1} = x_{i+1} - x_i, \qquad i = 1, \dots, n.
\]
Thus all subinterval lengths are equal. Since the total length is \(b-a\), we have
\[
x_i - x_{i-1} = \frac{b-a}{n+1}, \qquad i = 1, \dots, n+1.
\]
Hence,
\begin{equation}\label{pr2th1}
x_i = a + i \,\frac{b-a}{n+1}, \qquad i = 0, 1, \dots, n, n+1.
\end{equation}

To show that this point is a global minimum, we use Lemma \ref{lemma4} and Lemma~\ref{lemma5}.
By Lemma \ref{lemma4}, the Hessian matrix of \(S\) has the form
\[
H = |A| \cdot
\begin{pmatrix}
x_2 - x_0 & -(x_2 - x_1) & 0 & \cdots & 0 \\
-(x_2 - x_1) & x_3 - x_1 & -(x_3 - x_2) & \cdots & 0 \\
0 & -(x_3 - x_2) & x_4 - x_2 & \cdots & 0 \\
\vdots & \vdots & \vdots & \ddots & \vdots \\
0 & 0 & 0 & \cdots & x_{n+1} - x_{n-1}
\end{pmatrix}.
\]
By Lemma \ref{lemma5}, this matrix is positive definite for any partition.
Thus, all conditions of Proposition \ref{prop1} are satisfied:
\(S\) is defined on the convex open domain
\[
\Omega = \{ (x_1, \dots, x_n)\in \mathbb{R}^{n}: a < x_1 < \dots < x_n < b \},
\]
has a unique critical point $x^{*}=(x_1, \dots, x_n)$ with coordinates satisfying \eqref{pr2th1}, and \(H\) is positive definite throughout \(\Omega\).
Therefore, by Proposition \ref{prop1}, this critical point is the unique global minimum.

Substituting \eqref{pr2th1} into \eqref{pr1th1}, we obtain the minimum area:
\[
S_{\min} = \frac{|A|}{6} \sum_{i=1}^{n+1} \left( \frac{b-a}{n+1} \right)^3
= \frac{|A|}{6(n+1)^2}(b-a)^3.
\]

Thus, Theorem \ref{theorem1} is proved. 
\end{proof}

\section{The  case $x^m$}

This section considers the optimal piecewise linear approximation of polynomial functions \(f(x)=x^m\), \(m\in\mathbb{N}\), \(m\geq 2\), on the interval \([a,b]\), \(a\geq 0\). For an arbitrary number of interior nodes, a general formula for the area is obtained and a system of equations for the critical points is derived.

\begin{lemma}\label{lemma6}
Let $m \in \mathbb{N}$, $m \ge 2$, and let $0 \le p < q$. Consider the function $f(x) = x^m$ on the interval $[p, q]$. Denote by $\ell_{p,q}(x)$ the secant line connecting the points $(p, p^m)$ and $(q, q^m)$ on the graph of $f$. Then the equation of the secant is given by
\[
\ell_{p,q}(x)=\frac{q^{m}-p^{m}}{q-p}\,x+\frac{qp^{m}-pq^{m}}{q-p},
\qquad 0 \le p < q.
\]
\end{lemma}

\begin{proof}
The equation of the line passing through the two points $(p, p^m)$ and $(q, q^m)$ has the form
\[
\ell_{p,q}(x) = p^m + \frac{q^m - p^m}{q - p}(x - p).
\]
Expanding the brackets and collecting like terms yield
\[
\ell_{p,q}(x) = \frac{q^m - p^m}{q - p} x + \frac{ q p^m-p q^m }{q - p}.
\]
\end{proof}

\begin{lemma}\label{lemma7}
Let $m \in \mathbb{N}$, $m \ge 2$, and let $0 \le p < q$. Consider the function $f(x) = x^m$ on the interval $[p, q]$. Denote by $\ell_{p,q}(x)$ the secant line connecting the points $(p, p^m)$ and $(q, q^m)$ on the graph of $f$.

Then the area $S(p,q)$ between the graph of $f(x) = x^m$ and the secant $\ell_{p,q}(x)$ on the interval $[p, q]$ is given by the formula:
\[
S(p,q) =
\begin{cases}
\dfrac{p^m q - q^m p}{2} + \dfrac{m-1}{2(m+1)}(q^{m+1} - p^{m+1}), & p > 0, \\[1.5ex]
\dfrac{m-1}{2(m+1)}\, q^{m+1}, & p = 0.
\end{cases}
\]
\end{lemma}

\begin{proof}
By Lemma \ref{lemma6}, the secant line connecting the points $(p, p^m)$ and $(q, q^m)$ on the graph of $f$ is given by
\[
\ell_{p,q}(x)=\frac{q^{m}-p^{m}}{q-p}\,x+\frac{qp^{m}-pq^{m}}{q-p}.
\]
Then the area between the graph and the secant is the difference of integrals:
\[
S(p,q) = \int_p^q \ell_{p,q}(x)\,dx - \int_p^q x^m\,dx.
\]

The integral of the secant equals the area of the trapezoid:
\[
\int_p^q \ell_{p,q}(x)\,dx = \frac{p^m + q^m}{2}(q-p).
\]

The integral of $x^m$ is:
\[
\int_p^q x^m\,dx = \frac{q^{m+1} - p^{m+1}}{m+1}.
\]

Hence,
\[
S(p,q) = \frac{p^m + q^m}{2}(q-p) - \frac{q^{m+1} - p^{m+1}}{m+1}.
\]

Expanding and simplifying yield:
\[
S(p,q) = \frac{p^m q - q^m p}{2} + \frac{m-1}{2(m+1)}(q^{m+1} - p^{m+1}), \qquad p > 0.
\]

For $p = 0$, the formula gives:
\[
S(0,q) = \frac{m-1}{2(m+1)}\, q^{m+1}.
\]

Thus, Lemma \ref{lemma7} is proved. 
\end{proof}

\begin{lemma}\label{lemma8}
Let \( m \in \mathbb{N} \), \( m \geq 2 \), and let \( 0 \leq a < b \).
Consider a partition of the interval \([a, b]\):

\[
a = x_0 < x_1 < \cdots < x_n < x_{n+1} = b.
\]

Let \( S(x_1, \dots, x_n) \) be the area between the graph of the function
\( f(x) = x^m \) and the inscribed polygonal line passing through the points

\[
(x_0, x_0^m), \; (x_1, x_1^m), \; \dots, \; (x_n, x_n^m), \; (x_{n+1}, x_{n+1}^m).
\]

Then the following formula holds:

\[
S=S(x_1, \dots, x_n) =
\frac{1}{2} \sum_{i=0}^{n} \bigl( x_i^m x_{i+1} - x_{i+1}^m x_i \bigr)
+ \frac{m-1}{2(m+1)} \bigl( b^{m+1} - a^{m+1} \bigr).
\]

\end{lemma}

\begin{proof}
On each subinterval \([x_i, x_{i+1}]\), where \( i = 0, 1, \dots, n \),
the polygonal line coincides with the secant line to the graph of
\( f(x) = x^m \). By Lemma \ref{lemma7}, the area between the graph and this secant is

\[
S(x_i, x_{i+1}) =
\frac{x_i^m x_{i+1} - x_{i+1}^m x_i}{2}
+ \frac{m-1}{2(m+1)} \bigl( x_{i+1}^{m+1} - x_i^{m+1} \bigr).
\]

The total area is the sum over all subintervals:

\[
S=S(x_1, \dots, x_n) = \sum_{i=0}^{n} S(x_i, x_{i+1}).
\]

Therefore,
\begin{equation}\label{pr1lemma8}
\begin{aligned}
S
= \frac{1}{2} \sum_{i=0}^{n} \bigl( x_i^m x_{i+1} - x_{i+1}^m x_i \bigr) 
 + \frac{m-1}{2(m+1)} \sum_{i=0}^{n} \bigl( x_{i+1}^{m+1} - x_i^{m+1} \bigr).
\end{aligned}
\end{equation}

Computing the second sum:
\begin{equation}\label{pr2lemma8}
\sum_{i=0}^{n} \bigl( x_{i+1}^{m+1} - x_i^{m+1} \bigr)
= x_{n+1}^{m+1} - x_0^{m+1} = b^{m+1} - a^{m+1}.
\end{equation}

Finally, combining \eqref{pr1lemma8} and \eqref{pr2lemma8} yields the desired formula. 
\end{proof}

Below we present the theorem on the critical points of the area function $S$ for the case $x^m$.

\begin{theorem}\label{theorem2}
Let \( f(x) = x^m \), \( m \in \mathbb{N} \), \( m \ge 2 \), and let
\[
a = x_0 < x_1 < \dots < x_n < x_{n+1} = b, \quad a \ge 0, 
\]
be a partition of the interval \([a, b]\).
Let \( S(x_1, \dots, x_n) \) be the area between the graph of \( f(x) \) and the inscribed polygonal line, as given in Lemma \ref{lemma8}. Then the critical points of \( S \) with respect to the variables \( x_1, \dots, x_n \) satisfy the following system:

\[
\frac{x_{k+1}^m - x_{k-1}^m}{x_{k+1} - x_{k-1}} = m x_k^{m-1}, \quad k = 1, \dots, n,
\]
where \( x_0 = a \), \( x_{n+1} = b \).
\end{theorem}

\begin{proof} 
By Lemma \ref{lemma8}, the area is
\[
S(x_1, \dots, x_n) =
\frac{1}{2} \sum_{i=0}^{n} \bigl( x_i^m x_{i+1} - x_{i+1}^m x_i \bigr)
+ \frac{m-1}{2(m+1)} \bigl( b^{m+1} - a^{m+1} \bigr).
\]

In the last expression, we separate the terms $$\alpha_k(x_{k-1},x_k, x_{k+1})=\frac{1}{2}\left(x_k(x^m_{k-1}-x^m_{k+1})+x^m_k(x_{k+1}-x_{k-1})\right)$$ containing $x_k$:
$$
S=\alpha_k(x_{k-1},x_k, x_{k+1})+ \bigl(S-\alpha_k(x_{k-1},x_k, x_{k+1})\bigr).
$$

Differentiating with respect to \(x_k\), \(k=1,\dots,n\), and setting the derivative to zero gives
\[
\frac{\partial S}{\partial x_k}=\frac{\partial \alpha_k }{\partial x_k}=   \frac{1}{2}\left(x^m_{k-1}-x^m_{k+1}+mx^{m-1}_k(x_{k+1}-x_{k-1})\right)
=0.
\]
After elementary transformations, we obtain the following system
\[
\frac{x_{k+1}^m-x_{k-1}^m}{x_{k+1}-x_{k-1}}=m x_k^{m-1},\quad k=1,\dots,n.
\]

Thus, Theorem \ref{theorem2} is proved.
\end{proof}

\section{ Optimal piecewise linear approximation of a cubic parabola}

In this section, we consider the optimal piecewise linear approximation of the cubic parabola $y=x^3$  with two interior nodes, for which we obtain exact formulas for the optimal nodes and the minimum area between the graph and the inscribed polygonal line.

\begin{theorem}\label{theorem3}
Let \(f(x)=x^3\), \(x\in[0,b]\), \(b>0\). Consider the partition
\[
0=x_0<x_1<x_2<x_3=b.
\]
Let \(S(x_1,x_2)\) be the area between the graph of \(x^3\) and the polygonal line connec\-ting the points \((x_i,x_i^3)\), \(i=0,1,2,3\). Then:
\[
x_1^*=\frac{1+\sqrt{33}}{16}\,b,\qquad
x_2^*=\frac{\sqrt{3}+3\sqrt{11}}{16}\,b,
\]
\[
S_{\min}=S(x_1^*,x_2^*)=
\frac{2048-117\sqrt{3}-495\sqrt{11}}{8192}\,b^4.
\]
\end{theorem}

\begin{proof}
By Lemma \ref{lemma8},
\begin{equation}\label{pr1th3}
S(x_1,x_2)=\frac12\left(x_1^3x_2-x_2^3x_1+x_2^3b-b^3x_2\right)+\frac14 b^4. 
\end{equation}
The partial derivatives are
\[
\frac{\partial S}{\partial x_1}=\frac12(3x_1^2x_2-x_2^3),\qquad
\frac{\partial S}{\partial x_2}=\frac12(x_1^3-3x_1x_2^2+3x_2^2b-b^3).
\]
Setting them to zero:
\begin{equation}\label{pr2th3}
\begin{cases}
3x_1^2x_2-x_2^3=0,\\
x_1^3-3x_1x_2^2+3x_2^2b-b^3=0.
\end{cases} 
\end{equation}
From the first equation, \(x_2=\sqrt{3}x_1\). Substituting into the second and letting \(t=x_1/b\):
\[
8t^3-9t^2+1=0 \quad\Rightarrow\quad (t-1)(8t^2-t-1)=0.
\]
The positive root is \(t=\frac{1+\sqrt{33}}{16}\). Hence,
\begin{equation}\label{pr3th3}
x_1^*=\frac{1+\sqrt{33}}{16}b,\qquad
x_2^*=\frac{\sqrt{3}+3\sqrt{11}}{16}b. 
\end{equation}
The Hessian matrix at the critical point is
\[
H=
\begin{pmatrix}
3\sqrt{3}(x_1^{*})^2 & -3(x_1^{*})^2 \\
-3(x_1^{*})^2 & 3\sqrt{3}x_1^{*}(b-x_1^{*})
\end{pmatrix}.
\]
\[
\Delta_1=3\sqrt{3}(x_1^{*})^2>0,\qquad
\Delta_2=\det H=9(x_1^{*})^3(3b-4x_1^{*})>0.
\]
By Sylvester's criterion, \(H\) is positive definite; hence \((x_1^*,x_2^*)\) is a local minimum. The system \eqref{pr2th3} has the unique solution \eqref{pr3th3}, and on the boundary \(S\to+\infty\), therefore, it is a global minimum.

Finally, substituting \eqref{pr3th3} into \eqref{pr1th3} yields
\[
S_{\min}=
\frac{2048-117\sqrt{3}-495\sqrt{11}}{8192}\,b^4.
\]
Theorem \ref{theorem3} is proved. 
\end{proof}

\begin{remark}
For $a=0,\ b=1$:
\[
x_1^*=\frac{1+\sqrt{33}}{16}\approx 0.42153516541,\qquad
x_2^*=\frac{\sqrt{3}+3\sqrt{11}}{16}\approx 0.73012032366,
\]
\[
S_{\min}=\frac{2048-117\sqrt{3}-495\sqrt{11}}{8192}\approx 0.02485605276.
\]
For the uniform partition $x_1=1/3,\ x_2=2/3$, the area is
\[
S_{\text{unif}}=\frac{1}{36}\approx 0.02777777778,
\]
which is larger than $S_{\min}$. Thus, the optimal non-uniform partition reduces the deviation area by approximately $10.5\%$.
\end{remark}

\section{Optimal approximation of a power function with one interior node}

The problem of optimal piecewise linear approximation of the power func\-tion \(f(x)=x^m\), \(m\in\mathbb{N}\), \(m\geq 2\), on the interval \([a,b]\), \(a\geq 0\), with one interior node is considered. Explicit formulas for the optimal node and the minimum area are obtained.

\begin{theorem}\label{theorem4}
Let \(f(x)=x^m\), \(m\in\mathbb{N}\), \(m\geq 2\), \(x\in[a,b]\), \(a\geq 0\). Consider the partition
\[
a=x_0<x_1<x_2=b.
\]
Let \(S(x_1)\) be the area between the graph of \(x^m\) and the polygonal line connecting the points \((a,a^m)\), \((x_1,x_1^m)\), \((b,b^m)\). Then:
\[
x_1^*=\left(\frac{b^m-a^m}{m(b-a)}\right)^{\frac{1}{m-1}},
\]
\[
S_{\min}=S(x_1^*)=\frac{m-1}{2}\left(\frac{b^{m+1}-a^{m+1}}{m+1}-(b-a)\left(\frac{b^m-a^m}{m(b-a)}\right)^{\frac{m}{m-1}}\right).
\]
\end{theorem}

\begin{proof}
By Lemma \ref{lemma8},
\begin{equation}\label{pr1th4}
S(x_1)=\frac12\left(a^m x_1 - x_1^m a + x_1^m b - b^m x_1\right)+\frac{m-1}{2(m+1)}(b^{m+1}-a^{m+1}). 
\end{equation}
Differentiating the function $S$, we obtain
\[
\frac{\partial S}{\partial x_1}=\frac12\left(a^m - m x_1^{m-1}a + m x_1^{m-1}b - b^m\right).
\]
Setting the latter expression equal to zero, we find the critical point
\begin{equation}\label{pr2th4}
x_1^*=\left(\frac{b^m-a^m}{m(b-a)}\right)^{\frac{1}{m-1}}. 
\end{equation}
The second derivative is
\[
\frac{\partial^2S}{\partial x_1^2}=\frac12 m(m-1)x_1^{m-2}(b-a)>0,
\]
so \(S(x_1)\) is strictly convex and \(x_1^*\) is the global minimum.

Finally, combining \eqref{pr2th4} and \eqref{pr1th4}, we have
\[
S_{\min}=\frac{m-1}{2}\left(\frac{b^{m+1}-a^{m+1}}{m+1}-(b-a)\left(\frac{b^m-a^m}{m(b-a)}\right)^{\frac{m}{m-1}}\right). 
\]
Theorem \ref{theorem4} is proved. 
\end{proof}

\begin{corollary}
For the cubic parabola on \([0,b]:\)
\[
x_1^*=\frac{b\sqrt{3}}{3},\qquad S_{\min}=\frac{9-4\sqrt{3}}{36}\,b^4.
\]
\end{corollary}


\begin{remark}
For the power functions $f(x)=x^m$, the optimal
interpolation nodes shift toward the right endpoint of the interval
when $m>2$, whereas for $m=2$ the partition is uniform. We prove this
for the case of one interior node.

Indeed, from Theorem \ref{theorem4} it follows that
$$
x_1^* = \left( \frac{b^m - a^m}{m(b-a)} \right)^{\frac{1}{m-1}}.
$$
For $m=2$, we have
$$
x_1^* = \frac{a+b}{2}\,.
$$
For $m>2$, it suffices to show that
\begin{equation}\label{pr1rem2}
\frac{b^m - a^m}{m(b-a)} > \left( \frac{a+b}{2} \right)^{m-1}. 
\end{equation}

By Jensen's inequality (see, e.g., Theorem 2.6.2 in \cite{Ransford}), for the convex function $\varphi(x)=x^{m-1}$
(since $m>2$) the following inequality holds:
$$
\frac{1}{b-a} \int_a^b x^{m-1}\,dx > \left( \frac{1}{b-a} \int_a^b
x\,dx \right)^{m-1} = \left( \frac{a+b}{2} \right)^{m-1}.
$$
Since
$$
\frac{1}{b-a} \int_a^b x^{m-1}\,dx = \frac{b^m - a^m}{m(b-a)},
$$
the inequality \eqref{pr1rem2} holds, hence $x_1^* > \dfrac{a+b}{2}$. 
\end{remark}

\section{Conclusions}

The proposed approach complements the classical theory of spline approximation by introducing a geometric optimization criterion based on the area between a function and its piecewise linear interpolant. In contrast to the traditional approximation criteria formulated in terms of uniform or integral norms, the proposed objective functional leads to a nonlinear optimization problem for determining the interpolation nodes. We believe that this approach may also be extended to higher-order spline approximations and other classes of interpola\-ting functions.

\bigskip

\bigskip

\textbf{Acknowledgement}

The work was supported by the National Research Foundation of Ukraine, Project number 2025.07/0014, Project name: ``Modern problems of Ma\-thematical Analysis and Geometric Function Theory''.

\bigskip

CONTACT INFORMATION

\medskip

O. O.~Pokutnyi \\Institute of Mathematics of NAS of Ukraine,
        Kyiv-4, 01601;\\National University of Kyiv-Mohyla Academy,
        Kyiv, 04070, Ukraine;\\lenasas@gmail.com

\medskip

R. R.~Salimov\\Institute of Mathematics of NAS of Ukraine,
        Kyiv-4, 01601, Ukraine;\\ruslan.salimov1@gmail.com

\medskip
M. V.~Stefanchuk\\Institute of Mathematics of NAS of Ukraine,
         Kyiv-4, 01601, Ukraine;\\stefanmv43@gmail.com

\end{document}